\documentclass[12pt]{amsart}
\usepackage{amssymb}

\usepackage[dvipdfm]{graphicx,color}

\numberwithin{equation}{section}  
\theoremstyle{plain}
\newtheorem{thm}{Theorem}[section]

\theoremstyle{definition}

\theoremstyle{remark}

\def\R{{\mathbb R}}

\def\Rn{{{\mathbb R}^n}}

\def\N{{\mathbb N}}

\def\FT{{\mathcal F}}
\def\L2tx{{L^2(\R_t\times\R^n_x)}}

\def\p#1{{\left({#1}\right)}}
\def\b#1{{\left\{{#1}\right\}}}

\def\n#1{{\left\|{#1}\right\|}}
\def\abs#1{{\left|{#1}\right|}}
\def\jp#1{{\left\langle{#1}\right\rangle}}

\def\supp{\operatorname{supp}}

\def\va{\varphi}

\title{Trace theorems: critical cases and best constants
}

\author[]{Michael Ruzhansky and Mitsuru Sugimoto}
\address{
  Michael Ruzhansky:
  \endgraf
  Department of Mathematics
  \endgraf
  Imperial College London
  \endgraf
  180 Queen's Gate, London SW7 2AZ, UK
  \endgraf
  {\it E-mail address} {\rm m.ruzhansky@imperial.ac.uk}
  \endgraf
  \medskip
 Mitsuru Sugimoto:
  \endgraf
  Graduate School of Mathematics
  \endgraf
  Nagoya University
  \endgraf
  Furocho, Chikusa-ku, Nagoya 464-8602, Japan
  \endgraf
  {\it E-mail address} {\rm sugimoto@math.nagoya-u.ac.jp}
  }
\thanks{The first
author was supported by the EPSRC Leadership
Fellowship EP/G007233/1.
}

\date{\today}

\begin{document}

\begin{abstract}
The purpose of this paper is to present the critical cases of the trace theorems for
the restriction of functions to closed surfaces, and to 
give the asymptotics for the norms of the traces under
dilations of the surface.
We also discuss the best constants for them.
\end{abstract}

\maketitle

\section{Introduction}

It is very well known that if $s>1/2$ and $\Sigma\subset\Rn$ is a closed 
hypersurface, $n\geq 2$, then we have the following trace theorem:
\begin{equation}\label{trace}
\n{f_{\,\,|\Sigma}}_{L^2\p{\Sigma\,;\,d\omega}}
\le
C \n{f}_{H^{s}(\R^n)},
\end{equation}
where $H^{s}$ is the Sobolev space over $L^{2}$ and $d\omega$ is
the induced surface measure on $\Sigma$.
It is also known that \eqref{trace} fails for $s=1/2$. 
The purpose of this note is to show how \eqref{trace} can be
modified to still hold for $s=1/2$. Moreover, if the origin belongs to the
set bounded by $\Sigma$, we investigate the dependence
of the constant $C$ in \eqref{trace} on a parameter $\rho\to\infty$
when we replace $(\Sigma; d\omega)$ by its dilation 
$(\rho\Sigma; \rho^{n-1} d\omega)$.
We also find best constants for some instances of \eqref{trace}.

We give a simple proof of such results by deriving 
estimates for different traces from the global smoothing estimates
for dispersive equations that have been established by the authors in \cite{RS3}, \cite{RS4}
by the geometric analysis using the methods of canonical transforms and
comparison principles developed  for the smoothing estimates.
We note that usually the argument is converse and one derives both
the smoothing estimates and the limiting absorption principle from
the appropriate trace theorem (see e.g. Ben-Artzi and Klainerman \cite{BK}). 
However, in this
instance, we show how new arguments and methods in PDEs can be 
applied to deduce facts about traces.

Let us formulate our results. 
Let
$a\in C^\infty\p{\R^n\setminus0}$ be real-valued and satisfy
$a(\xi)>0$ for all $\xi$. We will be restricting to the level set of the function $a$
defined by $\Sigma_a=\b{\xi\in\Rn\backslash 0: a(\xi)=1}$.
In order to simplify the exposition, let us modify the function $a$ outside
the set $\Sigma_{a}$ so that $a$ becomes positively homogeneous
of order two, namely, we can assume that $a$ is already positively
homogeneous satisfying
$a(\lambda \xi)=\lambda^2 a(\xi)$ for $\lambda>0$
and $\xi\neq0$.
The {dual hypersurface} $\Sigma^*_a$ is defined by
$\Sigma^*_a=\b{\nabla a(\xi):\xi\in\Sigma_a}$,
and the dual function $a^*(\xi)$ can be determined by the 
relation $\Sigma_{a^*}=\Sigma^*_a$. We discuss some of
its properties in Section \ref{S2}.

As usual, we denote by
$H^s(\Rn)$ and $\Dot{H}^s(\Rn)$ the Sobolev spaces
with the norms 
$\n{g}_{H^s(\Rn)}=\n{\jp{D_x}^sg}_{L^2(\Rn)}$ and
$\n{g}_{\Dot{H}^s(\Rn)}=\n{|D_x|^sg}_{L^2(\Rn)}$, respectively.
We use the notation $D_{x}=\frac{1}{i}\nabla_{x}$, so that
$|D_{x}|=\sqrt{-\Delta}$, $\jp{D_{x}}=\sqrt{1-\Delta},$ and for
a function $c(\xi)$ we denote by $c(D_{x})$ the Fourier multiplier
$c(D_{x})={\mathcal F}^{-1} c(\xi) {\mathcal F}$,
where ${\mathcal F}$ denotes the Fourier transformation
and ${\mathcal F}^{-1}$ its inverse defined by
\[
{\mathcal F}f(\xi)=\widehat{f}(\xi)=\int_{\Rn} e^{-ix\cdot\xi}f(x)\,dx,\quad
{\mathcal F}^{-1}f(x)={\widehat{f}}^*(x)=\frac1{(2\pi)^n}\int_{\Rn} e^{ix\cdot\xi}f(\xi)\,d\xi.
\]
\begin{thm}\label{Th:Fres}
Let
$a(\xi)\in C^\infty\p{\R^n\setminus0}$ be real-valued and satisfy
$a(\xi)>0$ and $a(\lambda \xi)=\lambda^2 a(\xi)$ for $\lambda>0$
and $\xi\neq0$. Let $\Sigma_a=\b{\xi\in\Rn\backslash 0: a(\xi)=1}$.
Suppose $s>1/2$. 
Then we have
\begin{equation}\label{Fres:1}
\n{f_{\,\,|\Sigma_a}}_{L^2\p{\Sigma_a\,;\,d\omega}}
\le
C \n{f}_{H^{s}(\R^n)}.
\end{equation}
Moreover, for $1/2<s<n/2$, we have
\begin{equation}\label{Fres:1.5}
\n{f_{\,\,|\Sigma_a}}_{L^2\p{\Sigma_a\,;\,d\omega}}
\le
C \n{f}_{\Dot{H}^{s}(\R^n)}.
\end{equation}
If we in addition assume that the Gaussian curvature of
$\Sigma_a$ is non-vanishing, then we have also
the critical cases
\begin{equation}\label{Fres:2}
\n{\p{\frac{\nabla a(x)}{|\nabla a(x)|}\wedge\frac {D_x}{|D_x|}}
f_{\,\,|\Sigma_a}}_{L^2\p{\Sigma_a\,;\,d\omega}}
\le
C \n{f}_{\Dot{H}^{1/2}(\R^n)}
\end{equation}
and
\begin{equation}\label{Fres:3}
\n{\p{\frac x{|x|}\wedge\frac {\nabla a^*(D_x)}{|\nabla a^*(D_x)|}}
f_{\,\,|\Sigma_a}}_{L^2\p{\Sigma_a\,;\,d\omega}}
\le
C \n{f}_{\Dot{H}^{1/2}(\R^n)},
\end{equation}
where $a^*(x)$ is the dual function of $a(\xi)$.
\end{thm}
Here the outer product $p\wedge q$ of vectors
$p=(p_1,p_2,\ldots,p_n)$ and $q=(q_1,q_2,\ldots,q_n)$
is defined by $p\wedge q=(p_iq_j-p_jq_i)_{i<j}$.
The condition that the Gaussian curvature of $\Sigma_{a}$ is 
non-vanishing can be expressed as $\det\nabla^{2}a(\xi)\not=0$ for
all $\xi\not=0.$

The third and fourth
estimates \eqref{Fres:2} and \eqref{Fres:3} in Theorem \ref{Th:Fres} say that
we can attain the critical order $s=1/2$ in the first and second estimates
\eqref{Fres:1} and \eqref{Fres:1.5} under a certain structure condition. 
The operators appearing in \eqref{Fres:2} and \eqref{Fres:3} are related to the
Laplace-Beltrami operator on $\Sigma_{a}$, but are of order
zero in both $x$ and $\xi$.
The precise geometric
meaning of the structure \eqref{Fres:2} and \eqref{Fres:3} will be related to the
Hamiltonian flow of the evolution governed by the defining function $a(\xi)$.

We also get the global version of the estimates, so that Theorem
\ref{Th:Fres} follows by setting $\rho=1$
from the following:

\begin{thm}\label{Th:Fres2}
Assume conditions of Theorem \ref{Th:Fres}. Then we have 
the uniform trace estimates
\begin{equation}\label{EQ:nonhom}
\n{f_{|\rho\Sigma_a}}_{L^2(\rho\Sigma_a;\rho^{n-1}
d\omega)} \leq C \n{f}_{H^s(\Rn)}
\qquad (s>1/2),
\end{equation}
and
\begin{equation}\label{EQ:gl}
\n{f_{|\rho\Sigma_a}}_{L^2(\rho\Sigma_a;\rho^{n-1}
d\omega)} \leq C \rho^{s-1/2} \n{f}_{\Dot{H}^s(\Rn)}
\qquad (n/2>s>1/2),
\end{equation}
for any $\rho>0$.
Moreover, in the critical cases, we have
\begin{align}\label{critical}
&\n{\p{\frac{\nabla a(x)}{|\nabla a(x)|}\wedge\frac {D_x}{|D_x|}}
f_{|\rho\Sigma_a}}_{L^2(\rho\Sigma_a;\rho^{n-1}
d\omega)}  \leq
C \n{f}_{\Dot{H}^{1/2}(\Rn)},
\\
\label{critical2}
&\n{\p{\frac x{|x|}\wedge\frac {\nabla a^*(D_x)}{|\nabla a^*(D_x)|}}
f_{|\rho\Sigma_a}}_{L^2(\rho\Sigma_a;\rho^{n-1}
d\omega)}  \leq
C \n{f}_{\Dot{H}^{1/2}(\Rn)},
\end{align}
with constants independent of $\rho>0$.
\end{thm}

We note that estimates \eqref{Fres:1.5}, \eqref{Fres:2} and \eqref{Fres:3}
are equivalent to estimates
\eqref{EQ:gl}, \eqref{critical} and \eqref{critical2} with the same constants
$C$, respectively.
Indeed, \eqref{EQ:gl}, \eqref{critical} and \eqref{critical2} follow
from estimates \eqref{Fres:1.5}, \eqref{Fres:2} and \eqref{Fres:3}
by just setting $f=f_\rho$, where the notation $g_\rho$ denotes the dilation
$g_\rho(x)=g(\rho x)$.
The converse is trivial.

Furthermore, we note that we get Theorem \ref{Th:Fres2} with
the critical case from
the following generalised result:

\begin{thm}\label{cor:sigma}
Let
$a(\xi)\in C^\infty\p{\R^n\setminus0}$ be real-valued and satisfy
$a(\xi)>0$ and $a(\lambda \xi)=\lambda^2 a(\xi)$ for $\lambda>0$
and $\xi\neq0$. Assume that the Gaussian curvature of the set 
$\Sigma_a=\b{\xi\in\Rn\backslash 0: a(\xi)=1}$ never vanishes.
Let a pseudo-differential operator $\sigma(X,D)$
have symbol $\sigma(x,\xi)$ which is  smooth in $x\neq0$, $\xi\neq0$,
and which is 
positively homogeneous of order $\beta$ in $x$ and of order $\alpha$ in $\xi$,
i.e. $\sigma(\lambda x,\mu\xi)=\lambda^{\beta}\mu^{\alpha}\sigma(x,\xi)$ for
all $x\not=0, \xi\not=0, \lambda>0,\mu>0.$
Suppose also the structure condition
\begin{equation}\label{EQ:gamma-atau2}
\begin{aligned}
 & \sigma(x,\lambda\nabla a(x))=0\quad \text{for all}\quad x\neq0
 \textrm{ and } \lambda\in\R,
\\
\text{or}\quad&
\\ 
 & \sigma(-x,\lambda\nabla a(x))=0\quad \text{for all}\quad x\neq0
 \textrm{ and } \lambda\in\R.
\end{aligned}
\end{equation}
Then we have 
the estimate
\begin{equation}\label{EQ:tau-2}
\n{\p{\sigma(X,D)
f}_{|\rho\Sigma_a}}_{L^2(\rho\Sigma_a;\rho^{n-1}
d\omega)} 
\leq
C \rho^{\beta}\n{f}_{\Dot{H}^{1/2+\alpha}(\Rn)},
\end{equation}
with a constant $C$ independent of $\rho>0.$
In particular, if $\alpha=\beta=0$,
we obtain the uniform critical estimate
\begin{equation}\label{EQ:tau-3}
\n{\p{\sigma(X,D)
f}_{|\rho\Sigma_a}}_{L^2(\rho\Sigma_a;\rho^{n-1}
d\omega)}  
 \leq
C \n{f}_{\Dot{H}^{1/2}(\Rn)}.
\end{equation}
\end{thm}
Again we remark that estimates \eqref{EQ:tau-2} and \eqref{EQ:tau-3} are
equivalent to themselves with $\rho=1$, with the same constant $C$,
respectively.
And we also remark that two types of conditions in
\eqref{EQ:gamma-atau2} which have different sign coincide with each other
if $a(\xi)$ is not only positively
homogeneous but homogeneous.

In the next section we explore the geometric meaning of operators in 
\eqref{Fres:2} and \eqref{Fres:3} together with that of the structure
condition \eqref{EQ:gamma-atau2}
in more detail and show how these estimates
can be deduced from the smoothing estimates for suitable
evolution equations. For different results on smoothing estimates
we refer to the first papers
\cite{Ka2, BD2, CS, Sj, V}, and to the authors' paper \cite{RS4} for the
setting of evolution equations corresponding to this paper.

Finally, we remark that we can use the argument of this paper to also
conclude some best constants.
We give the statement corresponding
to the case of the sphere, $a(\xi)=|\xi|^{2}$, in \eqref{Fres:1.5}.
\begin{thm}\label{cor:bc}
Let $n/2>s>1/2$. Then we have the estimate
\begin{equation}\label{EQ:glbc}
\n{f_{|\mathbb S^{n-1}}}_{L^2(\mathbb S^{n-1};
d\omega)} \leq 
\bigg(2^{1-2s} \frac{\Gamma(2s-1) \Gamma(\frac{n}{2} -s)}{\Gamma(s)^2 \Gamma(\frac{n}{2} - 1+s)} \bigg)^{1/2}
\n{f}_{\Dot{H}^s(\Rn)},
\end{equation}
and the constant in estimate
\eqref{EQ:glbc} can not be improved.
\end{thm}

We can also conclude a more general
result with a possibility to put different derivatives and weights.

\begin{thm}\label{cor:bc2}
Let $\sigma:(0,\infty)\to\R$ and $w:(0,\infty)\to\R$ be functions such that
the following expression is finite:
\begin{equation}\label{EQ:C0}
C_{1}=
\p{\mathop{\sup_{t>0}}_{k\in\N}
\b{\frac1{\sigma(t)^2}
\int_0^\infty J_{\nu(k)}(rt)^2 \frac{r}{w(r)^{2}}dr}}^{1/2}<\infty,
\end{equation}
where for $\lambda>-1/2$ the Bessel function 
$J_\lambda$ of order
$\lambda$ is given by
\begin{equation}\label{EQ:BF}
J_\lambda(t)=\frac{t^\lambda}{2^\lambda
\Gamma(\lambda+1/2) \Gamma(1/2)} \int_{-1}^1
e^{it r} (1-r^2)^{\lambda-1/2} dr,
\end{equation}
and $\nu(k)=n/2+k-1$.
Then we have the estimate
\begin{equation}\label{EQ:glbc2}
\n{f_{|\rho\mathbb S^{n-1}}}_{L^2(\rho\mathbb S^{n-1};\rho^{n-1}
d\omega)} \leq 
C_{1} \sqrt{\rho}\,\sigma(\rho)\n{w(|D_x|)f}_{L^2(\Rn)},
\end{equation}
for all $\rho>0$,
and the constant in estimate
\eqref{EQ:glbc2} can not be improved.
\end{thm}
In particular, if we take $\sigma(t)=t^{s-1}$ and $w(t)=t^s$,
the result is equivalent to Theorem \ref{cor:bc}.
If we take $\sigma(t)=t^{-1/2}$ and $w(t)=(1+t^2)^{s/2}$ ($s>1/2$),
the result gives the best constant of estimate \eqref{EQ:nonhom}
in Theorem \ref{Th:Fres2} in the case $a(\xi)=|\xi|^2$.

\section{Smoothing estimates and trace theorems}
\label{S2}

Let us start with the result that was
established by the authors in \cite{RS3}.
This concerns the critical case
($\alpha=1/2$) of the Kato--Yajima's estimate
\begin{equation}\label{EQ:critical-est1}
\n{|x|^{\alpha-1}|D_x|^{\alpha}e^{ita(D_x)}\varphi(x)}_
{L^2\p{\R_t\times\R^n_x}}
\leq C\n{\varphi}_{L^2\p{\R^n_x}},
\end{equation}
which holds for $1-n/2<\alpha<1/2$
(see \cite{KY} for $0\leq \alpha<1/2$ and \cite{Su1} for $1-n/2<\alpha<1/2$
in the case $a(D_x)=-\Delta_x$.
For general $a(D_x)$, see \cite[Theorem 5.2]{RS4}). 
It was shown in
\cite{RS4} that this estimate for
values of $\alpha$ close to $1/2$ implies the same estimate
for smaller $\alpha$. Thus, the
critical case of this estimate with $\alpha=1/2$ is important,
especially since it can be applied to the well-posedness problems
of the derivative nonlinear Schr\"odinger equations
(see \cite{RS5}). However, the estimate fails in the critical
case $\alpha=1/2$
(see Watanabe \cite{W}, or the authors' paper \cite{RS4} for more general negative
results) 
and it is known that it is necessary to cut-off the
radial derivatives for the estimate to hold in the critical
case as well (see \cite{Su2} or \cite{RS4}). This can be done by
replacing operator
$|D_x|^\alpha$ by the Laplace-Beltrami operator on
the sphere at the level $\alpha=1/2$. In fact, it turns out
one can use any operator as long as
its symbol vanishes on a certain set related to the symbol
of the Laplace operator (the sphere is this case).
To explain
this precisely, let us formulate it for the equation
\begin{equation}\label{eq-cr1}
\left\{
\begin{aligned}
\p{i\partial_t+a(D_x)}\,u(t,x)&=0,\\
u(0,x)&=\varphi(x)\in L^2\p{\R^n_x},
\end{aligned}
\right.
\end{equation}
where the real-valued function
$a=a(\xi)\in C^\infty\p{\R^n\setminus0}$ 
is elliptic and positively homogeneous of order two,
that is, it satisfies
$a(\xi)>0$ and $a(\lambda \xi)=\lambda^2 a(\xi)$ for $\lambda>0$
and $\xi\neq0$.
We remark that under these condition we have
the dispersiveness, namely,
$\nabla a(\xi)\neq0$ by the Euler's identity
$a(\xi)=1/2\nabla a(\xi)\cdot\xi$ and the ellipticity of $a(\xi)$.
The case $a(\xi)=|\xi|^2$ corresponds to the usual Laplacian
$a(D_x)=-\Delta_x$.
\par
Let us define
$\b{(x(t),y(t)):\,t\in\R}$ to be the classical orbit, 
that is, the solution
of the Hamilton-Jacobi ordinary differential equations
\[
\left\{
\begin{aligned}
\dot{x}(t)&=\p{\nabla_\xi a}(\xi(t)), \quad\dot{\xi}(t)=0,
\\
x(0)&=0,\quad \xi(0)=\xi_0,
\end{aligned}
\right.
\]
and consider the set of the paths of all
classical orbits
\begin{equation}\label{Ga}
\begin{aligned}
\Gamma_a
&=\b{\p{x(t),\xi(t)}\,:\, t\in\R,\, \xi_0\in\R^n\setminus 0}
\\
&=\b{\p{\lambda\nabla a(\xi),\xi}\,:\,
 \xi\in\R^n\setminus 0,\,\lambda\in\R}.
\end{aligned}
\end{equation}
Let a pseudo-differential operator $\sigma(X,D)$
have symbol $\sigma(x,\xi)$ which is  smooth in $x\neq0$, $\xi\neq0$,
and which is 
positively homogeneous of order $-1/2$ with respect to $x$,
and of order $1/2$ with respect to $\xi$.
Suppose also the structure condition
\begin{equation}\label{EQ:gamma-a}
 \sigma(x,\xi)=0\quad \text{if}\quad (x,\xi)\in
 \Gamma_a\quad \text{and}\quad x\neq0.
\end{equation}
Then it was shown in \cite{RS3} that the solution 
$u=e^{ita(D_x)}\varphi$ to 
$(\ref{eq-cr1})$ satisfies
\begin{equation}\label{math-ann}
\n{\sigma(X,D_x) e^{ita(D_x)}\varphi(x)}_{L^2\p{\R_t\times\R^n_x}}
\leq C
\n{\varphi}_{L^2(\R^n)}
\end{equation}
if $n\geq2$ and the Gaussian curvature of
the hypersurface
\begin{equation}\label{hyper}
\Sigma_a=\b{\xi\in\Rn\backslash 0\,:\, a(\xi)=1}
\end{equation}
never vanishes. We note that the set \eqref{Ga} and the assumption
\eqref{EQ:gamma-a} are somehow related to Sommerfeld's 
radiation condition, see \cite{RS5}.
We also note that condition \eqref{EQ:gamma-atau2} in Theorem \ref{cor:sigma}
means that
$\check\sigma(x,\xi)$ or $\check\sigma(x,-\xi)$ satisfies
the structure condition \eqref{EQ:gamma-a},
where $\check\sigma(x,\xi)=\sigma(\xi,x)$
The typical example for such operator $\sigma(X,D_x)$ is given by
the elements of 
\begin{equation}\label{eqex1}
\Omega_1=|x|^{-1/2}
\p{\frac x{|x|}\wedge\frac{\nabla a(D_x)}{|\nabla a(D_x)|}}|D_x|^{1/2}.
\end{equation}
Another interesting example is the element of
\begin{equation}\label{EQ:LB-dual}
\Omega_2=|x|^{-1/2}
\p{\frac{\nabla a^*(x)}{|\nabla a^*(x)|}\wedge\frac{D_x}{|D_x|}}|D_x|^{1/2},
\end{equation}
where $a^*(x)$ is the { dual function} of $a(\xi)$ which is positively
homogeneous of order two and is characterised by the
relation $a^*(\nabla a(\xi))=1$.
We remark that the sum of the squares of all
elements of $\Omega_2$ forms the
main factor of the homogeneous
extension of the Laplace-Beltrami operator on the {dual hypersurface}
$\Sigma^*_a=\b{\nabla a(\xi):\xi\in\Sigma_a}$.
The dual function $a^*(x)$ can be also determined by the 
relation $\Sigma_{a^*}=\Sigma^*_a$.

Let us say a few words regarding the dual hypersurfaces as it is
shown in \cite[Theorem 3.1]{RS3}. Suppose $n\geq2$. For the convenience of the
formulation, in order to have the gradient of the function to be homogeneous
of order zero, we may consider an extension of a function from its fixed
level set as a positively homogeneous function of order one.
Thus, let $p\in C^\infty\p{\R^n\setminus0}$ be a positive and
positively homogeneous function of order one.
We assume that the Gaussian curvature of the hypersurface
$\Sigma_p=\{\xi:p(\xi)=1\}$ never vanishes and we denote
$\Sigma_p^{*}=\{\nabla p(\xi): \xi\in\Sigma_{p}\}.$
Then there exists a unique positive and positively homogeneous
function $p^*\in C^\infty\p{\R^n\setminus0}$ of order one
such that
$\Sigma_p^*=\Sigma_{p^*}$,
$\Sigma_{p^*}^*=\Sigma_p$,
and the Gaussian curvature of the hypersurface
$\Sigma_{p^*}$ never vanishes.
Moreover, $\nabla p:\Sigma_p\to\Sigma_{p^*}$ is a $C^\infty$-diffeomorphism
and $\nabla p^*:\Sigma_{p^*}\to\Sigma_p$ is its inverse.

The proof of Theorem \ref{Th:Fres2}
 relies on the critical case of the
limiting absorption principle which can be proved by reducing
its statement to a model situation by the canonical transform
method combined with weighted estimates for the transform
operators.
On the other hand, it can be reduced to a corresponding
smoothing estimate for the Laplace operator with {\em any}
critical operator, for example to the homogeneous extension
of the Laplace-Beltrami operator on the sphere, 
recovering, in particular, the result of \cite{Su2}. 
This particular case has been extended to
include small perturbations by Barcel\'o, Bennett and Ruiz \cite{BBR}.
For further details on these arguments we refer to authors'
paper \cite{RS3}. At the same time, the set $\Gamma_a$ 
corresponds to the Hamiltonian flow of $a(D_x)$, which is known
to play a role in such problems also in a more general setting
of manifolds. There, non-trapping conditions also enter
(e.g. Doi \cite{Do1, Do2} in the case of Schr\"odinger operators
on manifolds, using Egorov theorem, or 
Burq \cite{Bu} and Burq, G\'erard and 
Tzvetkov \cite{BGT} in the case of Schr\"odinger
boundary value problems, using propagation properties of
Wigner measures), and such conditions can be also
expressed in terms of properties of the set $\Gamma_a$. In our
case this simply corresponds to the dispersiveness of $a(D_x)$.
We also record the result for non-homogeneous weights, namely,
\begin{equation}\label{EQ:BK}
\n{\jp{x}^{-s}|D_x|^{1/2}e^{ita(D_x)}\varphi(x)}_
{L^2\p{\R_t\times\R^n_x}}
\leq C\n{\varphi}_{L^2\p{\R^n_x}},
\end{equation}
which holds for all $s>1/2$. For the Schr\"odinger case 
$a(\xi)=|\xi|^{2}$ and $n\geq 3$
this was shown by Ben-Artzi and Klainerman \cite{BK}, and it was
extended in \cite{RS4}, in particular,
to any homogeneous of order 2 function $a(\xi)$ with
$\nabla a(\xi)\not=0$ for $\xi\not=0$, $n\geq 1$.
We will now prove Theorem \ref{Th:Fres2}.

\begin{proof}[Proof of Theorem \ref{Th:Fres2}]
First note that the formal adjoint
$T^*:\mathcal{S}(\R_t\times\R^n_x)\to\mathcal{S}'(\R^n_x)$
of the operator
\[
T=e^{ita(D_x)}:\mathcal{S}(\R^n_x)\to\mathcal{S}'(\R_t\times\R^n_x),
\]
the solution operator 
to equation \eqref{eq-cr1}, is expressed as
\begin{equation}\label{rel:fr}
T^*\left[v(t,x)\right]=\FT^{-1}_{\xi}
  \left[
    \p{\FT_{t,x}v}\p{a(\xi),\xi}
  \right].
\end{equation}
Then, we claim that for any operator $A=A(X,D_x)$ acting on the variable $x$, 
the estimate
\begin{equation}\label{Smest}
\n{Ae^{ita(D_x)}\varphi}_{L^2\p{\R_t\times\R^n_x}}
\le
C_{0}\,
\Vert \,\varphi\,\Vert_{L^2\p{\R^n_x}}
\end{equation}
is equivalent to the estimate
\begin{equation}\label{est:restrict}
\begin{aligned}
\n{\widehat{A^*f}_{|\rho\Sigma_a}}_
{L^2\p{\rho\Sigma_a\,;\,2\rho^{n-1}d\omega/|\nabla a(\omega)|}}
&\le
\frac{(2\pi)^{n/2}}{\sqrt{\pi}} C_{0}\sqrt{\rho}\,
\n{f}_{L^2(\R^n)}
\\
&=
\frac{1}{\sqrt{\pi}} C_{0}\sqrt{\rho}\,
\|\widehat f\|_{L^2(\R^n)}
\quad\text{for all $\rho>0$},
\end{aligned}
\end{equation}
where $d\omega$ is the standard surface element of the hypersurface $\Sigma_a$
and $\rho\Sigma_a=\{\rho\omega:\,\omega\in\Sigma_a\}$.
Indeed, by \eqref{rel:fr} and Plancherel's theorem,
we have
\begin{equation}\label{TAdual}
\begin{aligned}
\n{T^*A^*v}^2_{L^2(\R^n)}
&=(2\pi)^{-n}\n{\p{\FT_{t,x}A^*v}(a(\xi),\xi)}^2_{L^2(\R^n_\xi)}
\\
&=(2\pi)^{-n}\int^\infty_0
  \p{\int_{\Sigma_a}
   \left|
     \p{\FT_{t,x}A^*v}\p{\rho^2,\rho\omega}
   \right|^2
    \,\frac{2\rho^{n-1}d\omega}{|\nabla a(\omega)|}}\,d\rho.
\end{aligned}
\end{equation}
Here we have used the change of variables $\xi\mapsto\rho\omega$
($\rho>0,\omega\in\Sigma_a$).
Then for $v(t,x)=g(t)f(x)$ we have
\[
\n{T^*A^*v}^2_{L^2(\R^n)}
=(2\pi)^{-n}\int^\infty_0\abs{\widehat g(\rho^2)\sqrt\rho}^2
  \p{\int_{\Sigma_a}
   \left|\frac1{\sqrt{\rho}}\p{\widehat{A^*f}}\p{\rho\omega}
   \right|^2
    \,\frac{2\rho^{n-1}d\omega}{|\nabla a(\omega)|}}\,d\rho.
\]
At the same time, by \eqref{Smest}, we have
\[
\n{T^*A^*v}^2_{L^2(\R^n)}
\leq C_{0}^{2}\n{v}^2_{L^2(\R_t\times\R^n_x)}
=C_{0}^{2}\n{g}_{L^2(\R)}^2\n{f}_{L^2(\R^n)}^2.
\]
Note that we have by Plancherel's theorem
\[
\n{g}_{L^2(\R)}^2=\frac1{2\pi}\n{\widehat g}_{L^2(\R)}^2=
\frac1{\pi}\int^\infty_0\abs{\widehat 
g(\rho^2)\sqrt\rho}^2  \,d\rho,
\]
if $\supp\widehat g\subset[0,\infty)$.
Combining all these relations and taking arbitrary $g$,
we obtain estimate \eqref{est:restrict}.

Conversely, noting that
$\p{\FT_{t,x}A^*v}\p{\rho^2,\rho\omega}=\widehat{A^*f}_{|\rho\Sigma_a}$,
where $f(\,\cdot\,)=\FT_{t}v(\rho^2,\,\cdot\,)$,
the right hand side of equality \eqref{TAdual}
can be estimated by using the first line of \eqref{est:restrict}, and we actually have 
\begin{align*}
\n{T^*A^*v}^2_{L^2(\R^n)}
&\le
\frac{C_{0}^2}{\pi} \int_0^\infty\int_{\R^n}
   \left|
     \p{\FT_tv}\p{\rho^2,x}
   \right|^2\rho
    \,d\rho dx
\\
&=
\frac{C_{0}^2}{2\pi} \int_0^\infty\int_{\R^n}
   \left|
     \p{\FT_tv}\p{\rho,x}
   \right|^2
    \,d\rho dx
\\
&\le
\frac{C_{0}^2}{2\pi} \n{\FT_tv}^2_{L^2(\R_t\times\R^n_x)}
=C_0^2\,\n{v}^2_{L^2(\R_t\times\R^n_x)}
\end{align*}
for any $v(t,x)\in L^2(\R_t\times\R^n_x)$,
which implies estimate \eqref{Smest}.
We remark that \eqref{est:restrict} implies a corresponding version of
the limiting absorption principle for $a(D_{x})$, which, in turn,
means that $A$ is $a(D_{x})$-supersmooth, see e.g. Kato \cite{Ka1} or
Ben-Artzi and Devinatz
\cite{BD1}. We will not discuss further details of this here as we do
not need it in this paper.

We have already reviewed examples of operators $A$ which satisfy
smoothing estimate \eqref{Smest}, hence the Fourier restriction
estimate \eqref{est:restrict}.
For example, using \eqref{EQ:BK} and \eqref{EQ:critical-est1},
we can take
\begin{equation}\label{specialop}
\begin{aligned}
&A_1=\jp{x}^{-s}|D_x|^{1/2}\qquad (s>1/2), \\
&A_2=|x|^{\alpha-1}|D_x|^{\alpha}\qquad (1-n/2<\alpha<1/2).
\end{aligned}
\end{equation}
We can also take $A=\sigma(X,D_x)$ which appeared
in estimate \eqref{math-ann}, especially the elements of the
operators $\Omega_1$ or $\Omega_2$ defined by \eqref{eqex1} or
\eqref{EQ:LB-dual},
but in this case we also
need the non-degenerate Gaussian
curvature condition on the hypersurface
$\Sigma_a$ defined by \eqref{hyper},
which is equivalent to $\det \nabla^2a(\xi)\not=0$ ($\xi\neq0$)
(see Miyachi \cite{Mi}, for example).
Their formal adjoints are given by
\begin{align*}
&A_1^*=|D_x|^{1/2}\jp{x}^{-s}\qquad (s>1/2),
\\
&A_2^*=|D_x|^{1-s}|x|^{-s}\qquad (n/2>s>1/2)
\end{align*}
(here we take $s=1-\alpha$ for $A_2$),
and also
\begin{align*}
&\Omega_1^*=
|D_x|^{1/2}
\p{\frac{\nabla a(D_x)}{|\nabla a(D_x)|}\wedge\frac x{|x|}}|x|^{-1/2},
\\
&\Omega_2^*=|D_x|^{1/2}
\p{\frac{D_x}{|D_x|}\wedge\frac{\nabla a^*(x)}{|\nabla a^*(x)|}}|x|^{-1/2}.
\end{align*}
Note that we have
$|\nabla a(\xi)|\geq C>0$ on $\Sigma_a$
since $\nabla a(\xi)\neq0$ ($\xi\neq0$) in our case.
From the construction, we have the same property
for $a^*$, as well.
We also note that
$\n{f}_{L^2_{s}(\R^n)}=(2\pi)^{-n/2}\|\widehat f\|_{H^s(\Rn)}$ and
$\n{f}_{\Dot{L}^2_{s}(\R^n)}=(2\pi)^{-n/2}\|\widehat f\|_{\Dot{H}^s(\Rn)}$,
for weighted $L^{2}$-spaces $L^{2}_{s}(\Rn)$ and $\Dot{L}^{2}_{s}(\Rn)$
defined by the norms
$\n{f}_{L^2_{s}(\R^n)}=\n{\jp{x}^{s}f}_{L^{2}(\Rn)}$ and
$\n{f}_{\Dot{L}^2_{s}(\R^n)}=\n{|x|^{s}f}_{L^{2}(\Rn)}$,
respectively.

Now, first, using \eqref{EQ:BK}, we get
\eqref{Smest} with operator $A_{1}$ in
\eqref{specialop}, from which,
on account of \eqref{est:restrict},
we can conclude the following trace result:
$$
\n{f_{|\rho\Sigma_a}}_{L^2(\rho\Sigma_a;\rho^{n-1}
d\omega)} \leq C \n{f}_{H^s(\Rn)}
\qquad (s>1/2).
$$
Here we remark that $\nabla a\neq0$ because of Euler's identity
$a(\xi)=(1/2)\xi\cdot\nabla a(\xi)>0$.
If we use $A_2$ in \eqref{specialop} instead, and estimate
\eqref{EQ:critical-est1},
we get
$$
\n{f_{|\rho\Sigma_a}}_{L^2(\rho\Sigma_a;\rho^{n-1}
d\omega)} \leq C \rho^{s-1/2} \n{f}_{\Dot{H}^s(\Rn)}
\qquad (n/2>s>1/2).
$$ 
In the critical cases, using estimate \eqref{math-ann} with
operators $\Omega_1$ or $\Omega_2$ defined by \eqref{eqex1} or
\eqref{EQ:LB-dual}, and formulae for their adjoints, we obtain
\begin{align*}
&\n{\p{\frac{\nabla a(x)}{|\nabla a(x)|}\wedge\frac {D_x}{|D_x|}}
f_{|\rho\Sigma_a}}_{L^2(\rho\Sigma_a;\rho^{n-1}
d\omega)}  \leq
C \n{f}_{\Dot{H}^{1/2}(\Rn)},
\\
&\n{\p{\frac x{|x|}\wedge\frac {\nabla a^*(D_x)}{|\nabla a^*(D_x)|}}
f_{|\rho\Sigma_a}}_{L^2(\rho\Sigma_a;\rho^{n-1}
d\omega)}  \leq
C \n{f}_{\Dot{H}^{1/2}(\Rn)}.
\end{align*}
This completes the proof.
\end{proof}

We now show that, in fact, the same argument yields a more general result.

\begin{proof}[Proof of Theorem \ref{cor:sigma}]
Let us assume the first part of condition \eqref{EQ:gamma-atau2}.
As for the case when we assume the second part,
the result is obtained straightforwardly from the relation
$$
\p{\sigma(X,D)f}(x)
=
\overline{(\overline\sigma(-X,D)(\overline f(-x)))(-x)}.
$$
Let us define $A(X,D)=|X|^{-1/2-\alpha}\check\sigma(X,D)|D_{x}|^{1/2-\beta}$,
where $\check\sigma(x,\xi)=\sigma(\xi,x)$
 Then 
its symbol $A(x,\xi)=|x|^{-1/2-\alpha}\check\sigma(x,\xi)|\xi|^{1/2-\beta}$ satisfies
$A(x,\xi)=0$ if $(x,\xi)\in
 \Gamma_a$ and $x\neq0, \xi\not=0.$ Moreover, the symbol
 $A(x,\xi)$ is homogeneous of order $-1/2$ in $x$ and of order
 $1/2$ in $\xi$. Consequently, 
 the operator $A(X,D)$ satisfies
 \eqref{EQ:gamma-a} and \eqref{math-ann}. The argument in
 the proof of Theorem \ref{Th:Fres2} yields
 the estimate \eqref{est:restrict}. 
 Taking $g$ so that ${\mathcal F}^{-1}\overline{g}(x)=|x|^{-1/2-\alpha}f(x)$,
 we obtain
 \begin{equation}\label{est:restrict2}
\rho^{1/2-\beta}\n{\mathcal F({\check\sigma(X,D)^{*}{\mathcal F}^{-1}\overline{g}})_{|\rho\Sigma_a}}_
{L^2\p{\rho\Sigma_a\,;\,\rho^{n-1}d\omega/|\nabla a|}}
\le
C\sqrt{\rho}\,
\n{g}_{\Dot{H}^{1/2+\alpha}(\R^n_x)}.
\end{equation}
We can now readily check that
\begin{equation}\label{taus}
\mathcal F({\check\sigma(X,D)^{*}{\mathcal F}^{-1}\overline{g}})(x)=
\overline{(\sigma(X,D)g)(x)}.
\end{equation}
Indeed, writing formally
$$
\p{\check\sigma(X,D)^{*}h}(x)={\mathcal F}_{\xi}^{-1}
\p{\int_{\Rn} e^{- i y\cdot\xi}\overline{\check\sigma(y,\xi)}\ h(y) dy}(x),
$$
we get
$$
\mathcal F\p{\check\sigma(X,D)^{*}{\mathcal F}^{-1}\overline{g}}(\eta)=
\int_{\Rn} e^{- i y\cdot\eta}\overline{\sigma(\eta,y)}{\mathcal F}^{-1}\overline{g}(y) dy,
$$
yielding \eqref{taus}.
Consequently, from estimate \eqref{est:restrict2} and
equality \eqref{taus} we obtain
\eqref{EQ:tau-2}. 

\end{proof}

%

On account of the argument in the proof of Theorem \ref{Fres:1},
it is apparent that the best constants of trace theorems for
the sphere $\mathbb S^{n-1}$
are obtained from those of smoothing estimates for Schr\"odinger case
$a(\xi)=|\xi|^2$.
\begin{proof}[Proof of Theorem \ref{cor:bc}]
The best constant $C_0$ of estimate \eqref{Smest} with
$A=A_2$ in \eqref{specialop} is
\begin{align*}
C_0
& =
 \bigg(\pi 2^{2\alpha-1} \frac{\Gamma(1-2\alpha) \Gamma(\frac{n}{2} +\alpha -1)}{\Gamma(1-\alpha)^2 \Gamma(\frac{n}{2} - \alpha)} \bigg)^{1/2}
\\
&=
 \bigg(\pi 2^{1-2s} \frac{\Gamma(2s-1) \Gamma(\frac{n}{2} -s)}{\Gamma(s)^2 \Gamma(\frac{n}{2} - 1+s)} \bigg)^{1/2},
\end{align*}
where $s=1-\alpha$
(see \cite{BS}).
We remark that this constant with $\alpha=0$, that is, $C_0=\sqrt{\pi/(n-2)}$
was given by an earlier work of
Simon \cite{Si}.
Then by the argument of the proof of Theorem \ref{Th:Fres2}, 
we have estimate \eqref{est:restrict} with the same $C_0$ as the best one.
Since it is equivalent to itself with $\rho=1$, we have the conclusion.
\end{proof}


\begin{proof}[Proof of Theorem \ref{cor:bc2}]
We recall that in general that
if $n\geq 2$ and $g$ is injective and
differentiable on $(0,\infty)$,
the best constant $C_0$ in the inequality
\begin{equation}\label{EQ:smW}
\n{w(|x|)^{-1}\sigma(|D_x|)^{-1}e^{it g(|D_x|)}\va(x)}_\L2tx\leq
C_0\n{\va}_{L^2(\R^n_x)}
\end{equation}
is given by
$$
C_0=\p{2\pi\mathop{\sup_{\rho>0}}_{k\in\N}
\b{\frac{\rho}{\sigma(\rho)^{2} g^\prime(\rho)}
\int_0^\infty J_{\nu(k)}(r\rho)^2 \frac{r}{w(r)^{2}} dr}}^{1/2},
$$
where for $\lambda>-1/2$ the Bessel function 
$J_\lambda$ of order
$\lambda$ is given by \eqref{EQ:BF} (see also \cite{BS}).
This expression for the best constants was obtained by Walther \cite{Wa2}, and it
can be used to analyse estimates for radially symmetric
equations by carefully looking at the asymptotic
behaviour of Bessel functions and subsequent integrals.
We now take $g(\rho)=\rho^{2}$ and
$A= w(|x|)^{-1} \sigma(|D_{x}|)^{-1}$. Estimate \eqref{EQ:smW} now implies 
\eqref{est:restrict} with $A^{*}=\sigma(|D_{x}|)^{-1} w(|x|)^{-1}.$
Consequently, we get
$$
\n{\sigma(|x|)^{-1}w(|D|)^{-1}f_{|\rho\mathbb S^{n-1}}}_{L^2(\rho\mathbb S^{n-1};\rho^{n-1}
d\omega)} \leq \frac1{\sqrt{\pi}}  C_{0} \sqrt{\rho}\n{f}_{L^{2}(\Rn)},
$$
which implies \eqref{EQ:glbc2}.
\end{proof}


\end{document}